\documentclass{article}
\usepackage{amscd,amsmath,amsfonts,amsthm,amssymb,srcltx}
\usepackage{hyperref}
\usepackage[all]{xy}
\usepackage{comment}
\usepackage[T1]{fontenc}
\usepackage[normalem]{ulem}
\usepackage{pst-text,pst-node,pst-coil}
\newtheorem{theorem}{Theorem}[section]
\newtheorem*{theorem*}{Theorem}
\newtheorem{lemma}[theorem]{Lemma}
\newtheorem{conj}[theorem]{Conjecture}
\newtheorem{proposition}[theorem]{Proposition}
\newtheorem{corollary}[theorem]{Corollary}
\newcommand{\End}{\mathrm{End}}
\newcommand{\Ext}{\mathrm{Ext}}
\newcommand{\Hom}{\mathrm{Hom}}

\newcommand{\im}{\mathrm{im}}

\newcommand{\amod}{{\text{\sf -Mod}}}

\newcommand{\amodf}{{\text{\sf -Mod}_f}}

\newcommand{\CC}{\mathbb{C}}

\newcommand{\A}{\mathcal{A}}

\newcommand{\pd}{\mathrm{pd}}

\newcommand{\har}{\mathrm{Harr}}
\newcommand{\hoch}{\mathrm{H}}
\newcommand{\id}{\mathrm{Id}}
\newcommand{\ps}{finitely unobstructed}

\newcommand{\ran}{\mathrm{Rep}_A^n}

\newcommand{\van}{V_n(A)}
\newcommand{\op}{\operatorname{op}}
\newcommand{\Set}{{\mathsf{Set}}}

\newcommand{\n}{\noindent}
\theoremstyle{definition}
\newtheorem{definition}[theorem]{Definition}
\newtheorem{example}[theorem]{Example}
\newtheorem{examples}[theorem]{Examples}

\theoremstyle{remark}
\newtheorem{remark}[theorem]{Remark}
\numberwithin{equation}{section}

\newcommand{\alg}{\mathcal{N}}
\newcommand{\calg}{\mathcal{C}}
\newcommand{\nn}{\mathcal{N}}

\begin{document}

\title{A new family of algebras whose representation schemes are smooth}

\author{Alessandro Ardizzoni\thanks{Partially supported by the research grant ``Progetti di Eccellenza 2011/2012'' from the ``Fondazione Cassa di Risparmio di Padova e Rovigo''. Member of GNSAGA.}, Federica  Galluzzi\thanks{Supported by the framework PRIN 2010/11 ``Geometria delle Variet\`a
Algebriche'', cofinanced by MIUR. Member of GNSAGA.}, Francesco Vaccarino\thanks{Supported by the Wallenberg grant. This work was set up during a visit of the last two authors to the Department of Mathematics, KTH (Stockholm, Sweden). Support by the Institut Mittag-Leffler (Djursholm, Sweden) is gratefully acknowledged.  Partially supported by the TOPDRIM project funded by the Future and Emerging Technologies program of the European Commission under Contract IST-318121 and by PRIN 2012 ``Spazi di Moduli e Teoria di Lie'' - \texttt{2012KNL88Y\_002}.}}
\date{}

\maketitle

\smallskip
\n

\allowdisplaybreaks
\begin{abstract} We give a necessary and sufficient smoothness condition for the scheme parameterizing the $n$-dimensional representations of a finitely generated associative algebra over an algebraically closed field. % of characteristic zero.
In particular, our result implies that the points $M \in \ran (k)$ satisfying $\Ext _A ^2(M,M)=0$ are regular. This generalizes well-known results on finite-dimensional algebras to finitely generated algebras.
\end{abstract}

%\tableofcontents

\section{Introduction}\label{intro}

Let $A$ be a finitely generated  associative $k$-algebra with $k$ an algebraically closed field.
Let $V_n(A)$ be the commutative $k$-algebra representing the functor from commutative algebras to sets
\[
\calg_k\to \Set:B\mapsto \Hom_{\alg_k}(A,M_n(B))
\]
of the $n$-dimensional representations of $A$ over $B,$ (see Section \ref{ran}).
The scheme $\ran$ of the linear representations of dimension $n$ of $A$ is defined to be $\mathrm{Spec}\,V_n(A).\,$

%In this paper we study the smoothness of the scheme $\ran.$

Formally smooth (or quasi-free) algebras provide a generalization of the notion of free algebra, since they behave like a free algebra with respect to  nilpotent extensions.
The definition goes back to J. Cuntz and D. Quillen and it was inspired by the Grothendieck's definition of formal smoothness given in the commutative setting, see \cite[19.3.1]{EGA4}.
See also \cite[19]{G1} and \cite[4.1.]{LB}. For further details, see \ref{subsec:FS}.

It is well-known that if $A$ is formally smooth then $\ran$ is smooth (see \cite[Proposition 19.1.4.]{G1} and \cite[Proposition 6.3.]{LB}).
If $A$ is finite-dimensional then it is formally smooth if and only if it is hereditary (see Theorem \ref{equiv}) and, therefore $\ran$ is smooth for all $n$ if and only if $A$ is hereditary (see \cite[Proposition 1]{Bo}).

For infinite-dimensional algebras the picture is more complex, e.g. there are hereditary algebras which are not formally smooth (see Remark \ref{infi}).
It is therefore interesting to find other sufficient (or necessary) conditions on $A$ which ensure $\ran$ to be smooth.

Let $M$ be an $A$-module in $\ran(k).$ It is well-known that the linear space $\Ext^2_{A}(M,M)\,$ contains the obstructions in extending the infinitesimal deformations of $M$ to the formal ones. For this reason an algebra $A$ such that $\Ext^2_A(M,M)=0,$ for all $M\in\ran(k)$ and $n\geq 1$, will be called {\emph{\ps}}.

It has been proved by Geiss and de la Pe\~{n}a (see \cite{Ge,Ge-P}) that, when $A$ is finite-dimensional, \ps \ implies that $\ran$ is smooth.

We underline that any hereditary algebra is \ps \ but the converse is not true, e.g. the universal enveloping algebra of a finite-dimensional semisimple Lie algebra is \ps \ but not hereditary if the dimension of the underlying Lie algebra is greater than one.

The proof given in \cite{Ge,Ge-P} is based on the analysis of the local geometry of $\ran,$ and it specifically relies on the upper semicontinuity of certain dimension functions arising from the bar resolution of $A.$ As we observe in the last section of this paper, their approach remains valid if one assumes that $A$ is finitely presented or bimodule coherent.
\bigskip

We follow here a different path, namely, we study the smoothness problem via the adjunction
\begin{equation}\label{adj}
\Hom_{\calg_{k}}(V_n(A),B)\xrightarrow{\cong}\Hom_{\nn_k}(A,M_n(B))
\end{equation}

The adjunction (\ref{adj}) allows us to use the Harrison cohomology of $V_n(A)$ instead of the Hochschild cohomology of $A$.
The Harrison cohomology of a commutative $k$-algebra is the symmetric part of its Hochschild cohomology, and it has been proved by Harrison \cite{Ha} that an affine ring $R$ is regular if and only if its second Harrison cohomology vanishes.

This is our main result.

\begin{theorem*}\label{B}
Let $A$ be a finitely generated $ k$-algebra, let $f:V_n(A)\rightarrow k$ be a $k$-algebra map and let $\rho:A\rightarrow M_n(k)$ be the algebra map that corresponds to $f$ through the adjunction above. Then there is a linear embedding of $\har^2(V_n(A),{_f k})$ into $H^2 (A, {_\rho M_n(k)_\rho})$.
As a consequence, $M\in \mathrm{Rep}_{A}^{n}$ is a regular point whenever $\Ext ^2_{A}(M,M)=0$.
\end{theorem*}

We have thus extended the known results on smoothness to infinite-dimensional finitely generated algebras.

We remark that the above embedding is not an isomorphism in general. We give a counterexample by using 2-Calabi Yau algebras  (Remark \ref{calabi}).

The paper goes as follows.

In paragraph \ref{ran} we recall the definition of $\ran$ as the scheme parameterizing the $n$-dimensional representations of $A.$

In Section \ref{harri} we recall the Harrison cohomology which may be seen as the commutative version of the Hochschild cohomology.
We prove that the regularity of a point in $\ran = \mathrm{Spec}(V_n(A))$ is equivalent to the vanishing of $\har^2(V_n(A),{_f k})$, for the $ k$-algebra map $f:V_n(A)\rightarrow k $ associated to the point (see Theorem \ref{teo:har}). Then Theorem \ref{thm:Hoch-Harr} shows that there is a linear embedding of $\har^2(V_n(A),{_f k})$ into $H^2 (A, {_\rho M_n(k)_\rho})$ and
as a consequence, that $M\in \mathrm{Rep}_{A}^{n}$ is a regular point whenever $\Ext ^2_{A}(M,M)=0.$

Then, as said before, by using 2-Calabi Yau algebras, we exhibit an example which shows that the above embedding is not an isomorphism.

In Section \ref{exapp} we present a list of examples and applications of the aforementioned results.
To this aim, we first recall the notions of formally smooth and  hereditary algebra. We mention the known result on the smoothness of $\ran$ when $A$ is formally smooth or hereditary to compare the notions of formally smoothness, hereditary, \ps {} and we stress the difference between the finite and the infinite-dimensional case.

Afterward, we give the definition of \ps {} algebra and we prove that if $A$ is \ps {} then $\ran $ is smooth (see  Corollary \ref{thm:ps}).

Then we produce examples of \ps \ algebras (neither hereditary nor formally smooth) whose associate representation scheme is smooth (see Example \ref{psex}).

In Section \ref{defsmooth} we study the relationships between the deformation theory of $M \in \ran(k),$ in the sense of Gerstenhaber, Geiss and de la Pe\~{n}a , and the deformation theory of $V_n(A)$ as usually defined in algebraic geometry.

In particular, by using the adjunction (\ref{adj}), we will see that there are no obstructions to the integrability of the infinitesimal deformations of $M$ if and only if $\har^2(V_n(A),{_f k})=0.$ Motivated by this fact we formulate the following conjecture.

\begin{conj}
The image of the embedding $\har^2(V_n(A),{_f k})\hookrightarrow \Ext^2(M,M)$ contains the subspace of $\Ext^2(M,M)$ of the obstructions to integrate the infinitesimal deformations of $M.$
\end{conj}

As a bonus, we further show that the approach to this smoothness problem developed  for $A$ finite-dimensional in \cite{Ge,Ge-P} works as well if $A$ is \ps \ and finitely presented or bimodule coherent.

\section{Preliminaries}

\subsection{Notations}\label{not} Unless otherwise stated we adopt the following
notations:
\begin{itemize}
\item $k$ is an algebraically closed field;
\item $F=k\{x_1,\dots,x_m\}$ is the associative free $k$-algebra on
$m$ letters;
\item $A\cong F/J$ is a finitely generated associative $k$-algebra;
\item $\alg_-,\,\calg_-$ and $\Set$ denote the
categories of $\ $-algebras,
commutative $\ $-algebras and sets, respectively;
\item The term "$A$-module" indicates a left $A$-module.   The categories of left -modules is denoted by $\amod$. The full subcategory of modules having finite dimension over $k$ will be denoted by $\amod_f;$
\item We write $\Hom_{\A}(B,C)$ for the morphisms from an object $B$ to $C$ in a category $\A$. If $\A=A\amod,\,$ then we will write $\Hom_{A}(-,-)$;
\item $A^{\op}$ is the opposite algebra of $A$ and $A^e:=A \otimes A^{\op}\,$ is the envelope of $A$. It
is an $A$-bimodule and a $k$-algebra. One can identify the category of the $A$-bimodules with $A^e\amod$ and we will do it thoroughly this paper;
\item $\mathrm{Ext}_{-}^i(\, , \,)$ denotes the $\Ext$ groups on the category $\amod ;$
\item $\hoch^i (A,-) $ is the Hochschild cohomology with coefficients in $A^e\amod .$
\end{itemize}

\subsection{The scheme of $n$-dimensional representations}\label{ran}

The study of the affine scheme $\ran$ of $n$-dimensional representations of an algebra $A$ goes back to the early $1970$'s with work of M. Artin, P. Gabriel, C. Procesi and D. Voigt. See for example \cite{Ga} and the references therein.

\noindent
Denote by $M_n(B)$ the full ring of $n \times n$ matrices over
$B,\,$ with $B$ a ring. If $f \ : \ B \to C $ is a ring homomorphism we
denote by $
M_n(f) :  M_n(B) \to M_n(C)$
the homomorphism induced on matrices.

\begin{definition}
Let  $A \in \alg_k,\, B \in \calg_k.\,$ By an {\em n-dimensional representation of} $A$ over $B$ we mean a homomorphism of $k$-algebras $
\rho\, : A \, \to M_n(B).\,
$
\end{definition}

\noindent
It is clear that this is equivalent to give an $A$-module structure on $B^n.$
The assignment $B\mapsto \Hom_{\calg_k}(A,M_n(B))$ defines a covariant functor
\begin{equation*}
\calg_{k} \longrightarrow \Set .
\end{equation*}
which is represented by a commutative $k$-algebra $V_n(A).\,$
\begin{lemma}\cite[Ch.4, \S 1]{Pr}\label{rep}
For all $A\in\nn_k$ and $\rho\, : A \, \to M_n(B)$ a linear representation, there exist $V_n(A)\in\calg_k$ and a
representation $\eta_A:A\to M_n(V_n(A))$ such that $\rho\mapsto
M_n(\rho)\circ \eta_A$ gives an isomorphism
\begin{equation}\label{proc}
\Hom_{\calg_{k}}(V_n(A),B)\xrightarrow{\cong}\Hom_{\nn_k}(A,M_n(B))\end{equation}
for all $B\in \calg_{k}$.
\end{lemma}

\smallskip
\noindent
If $A=F,\,$ one has that $V_n(F):=k[\xi_{lij}],\,$ the
polynomial ring in variables $\{\xi_{lij}\,:\, i,j=1,\dots,n,\, l=1,\dots,m\}$ over $k.\,$
If $A= F/J$ finitely generated $k$-algebra, one defines  $V_n(A):=k[\xi_{lij}]/I$  where
$I$ is the ideal of $V_n(F)$ generated by the $n \times n $ entries of $f(\xi_1, . . . ,\xi_m)$,
 $f$ runs over the elements of $J$ and $\xi_l$ is the matrix $(\xi_{lij})$. Therefore $V_n(A)$ is an affine ring (i.e. a finitely generated algebra with identity) when $A$ is a finitely generated $k$-algebra.

\begin{definition}\label{pigreco}
We write $\ran$ to denote Spec\,$V_n(A).\,$ It is considered as a
$k$-scheme.  The map
\begin{equation*}
\eta _A :A \to
M_n(V_n(A)),\;\;\;\;a_{l}\longmapsto \xi_{l}^A:=(\xi_{lij}+I).
\end{equation*}
is called {\em the universal n-dimensional representation.}
\end{definition}

\begin{examples}\label{commuting} (i) By construction, if $A=F,\,$ then $\mathrm{Rep}_F^n (k) = M_n(k)^m.\, $
If $A=F/J,\,$ the $B$-points of  $\ran$ can be described as follows:
\[
\ran(B)= \{(X_1,\dots,X_m) \in M_n(B)^m \ : \ f(X_1,\dots, X_m)=0\; \mbox{for all} \, f \in J \};
\]
(ii) If $A=\CC [x,y],$
$
\ran (\CC)  = \{(M_1,M_2)\in M_2(\CC)^2,M_1M_2 = M_2 M_1 \}
$
is the \textit{commuting scheme}, see \cite{V3}.
\end{examples}

\begin{remark} Note that $\ran$ may be quite complicated. It is not reduced in general and it seems to be hopeless to describe the coordinate ring of its reduced structure. The scheme $\ran$ is also known as {\it the scheme of} $n${\it -dimensional} $A${\it -modules}.
\end{remark}

\section{The main result}\label{harri}

We prove our main result using Harrison cohomology. Given a commutative ring $R$ and an $R$-module $N$, we denote by $\har^*(R,N)$ the Harrison cohomology group i.e. the group $\mathcal{E}^*(R,N)$ introduced in \cite{Ha}. Harrison cohomology can be seen as a commutative version of Hochschild cohomology. For further details the reader is referred to \cite[section 9.3]{We}, where $\har^2(R,N)$ is denoted by $\mathrm{H}^2_s(R,N)$. \medskip

The following standard result establishes a link between $\Ext^i_A(M,M)$ and the Hochschild cohomology of $A$ with coefficients in $\End_{k}(M)$.

\begin{theorem}\cite[Corollary 4.4.]{C-E}\label{hoch} We have
$$\Ext ^i_{A}(M,M)\cong \mathrm{H}^i(A, \End _{k} (M)).$$
\end{theorem}

For every algebra map $f:B\rightarrow A$ and $N\in A\amod$, denote by $_f N$ the corresponding left $B$-module structure on $N$. A similar notation is used on the right. In particular, if $N\in A^e\amod$, the notation $_fN_f$ means that $N$ is regarded as a $B^e$-module via $f$.

\begin{proposition}\label{pro:lifting}
The following assertions are equivalent for $A\in\alg_k$ and for every $M\in A^e\amod$:
 \begin{enumerate}
  \item  $\hoch^2(A,M)=0$;
  \item Let $f:A\to B$ be an algebra map and let $p:E\to B$ be a Hochschild extension of $B\in \alg_k$ with kernel $N$ such that $_fN_f=M$ (here $N$, being an ideal of square zero, is endowed with its canonical $B$-bimodule structure). Then $f$ has a lifting i.e. there is an algebra map $\overline{f}:A\to E$ such that $p\circ\overline{f}=f$.
  \begin{equation*}
  \xymatrixrowsep{15pt}
\xymatrix@C=30pt{&&&A\ar^{f}[d]\ar@{.>}_{\overline{f}}[dl]\\ 0 \ar[r] & N\ar[r]^i&E \ar@<0.5ex>@{->}^-{p}[r]  & B\ar@<0.5ex>@{..>}^-{\sigma}[l]\ar[r] &0  }
\end{equation*}
 \end{enumerate}
\end{proposition}

\begin{proof}
 The proof is the same of \cite[Proposition 9.3.3]{We} for our specific $M$. However, we recall a different proof of $(1)$ implies $(2)$ that will be needed in the proof of Theorem \ref{thm:Hoch-Harr}. Let $\omega: B\otimes B\to N$ be the Hochschild $2$-cocycle associated to the Hochschild extension $E$ of $B$ by $M$. Then $\overline{\omega}:=\omega\circ(f\otimes f):A\otimes A\to {_fN_f}=M$ is a Hochschild $2$-cocycle so that we can consider the Hochschild extension $A\oplus_{\overline{\omega}} M$ of $A$ by $M$, see \cite[page 312]{We}. Since $\overline{\omega}$ is, by assumption, a $2$-coboundary, then the latter extension is trivial i.e. there is an algebra map $s:A\to A\oplus _{\overline{\omega}} M$  which is a right inverse of the canonical projection. Composing $s$ with the algebra map $A\oplus _{\overline{\omega}} M\to E:(a,m)\mapsto \sigma f(a)+i(m)$ yields the required map $\overline{f}$.
\end{proof}

Let $R$ be a commutative noetherian ring. Recall that a point $p\in \mathrm{Spec}\, R$  is {\em regular} if the localization $R_p$ of $R$ at $p$ is a regular local ring i.e. $\dim_k(\mathfrak{m}/\mathfrak{m}^2)=\dim R_p,$
where $\mathfrak{m}$ is the unique maximal ideal of $R_p$ and $\dim R_p$ is its Krull dimension.
The ring $R$ is said to be {\em regular} if the localization at every prime ideal is a regular local ring.\\
The following result is a variant of \cite[Corollary 20]{Ha}.
\begin{theorem}\label{teo:har}
Let $f:V_n(A)\rightarrow k$ be a $ k$-point of $\ran.$ Then $f$ is a regular point of $\ran (k)$ if and only if $\, \har^2(V_n(A),{_f k})=0$.
\end{theorem}

\begin{proof} Set $\mathfrak{m}:=\ker(f)$ and $R:=V_n(A)$. Note that $k$ is a perfect field, as it is algebraically closed. Moreover, since $A$ is f.g., then $R$ is an affine ring (as observed after Lemma \ref{rep}) and hence we can apply \cite[Corollary 20]{Ha} to get that $f$ is regular if and only if $\har^2(R,R/\mathfrak{m})=0$. We conclude by observing that ${_f k}=\im f\cong R/\ker f=R/\mathfrak{m}$ as left $R$-modules.
\end{proof}

\begin{theorem}\label{thm:Hoch-Harr}
 Let $A$ be a f.g. $ k$-algebra, let $f:V_n(A)\rightarrow k$ be a $k$-algebra map and let $\rho:A\rightarrow M_n(k)$ be the algebra map that corresponds to $f$ through \eqref{proc}. Then there is a linear embedding of $\har^2(V_n(A),{_f k})$ into $H^2 (A, {_\rho M_n(k)_\rho})$.
As a consequence, $M\in \mathrm{Rep}_{A}^{n}$ is a regular point whenever $\Ext ^2_{A}(M,M)=0$.
 \end{theorem}

\begin{proof} Each $M\in \mathrm{Rep}_{A}^{n}$ is of the form $M\cong{_\rho (k^n)}$ for some $\rho:A\rightarrow M_n(k)$ as in the statement. By Theorem \ref{hoch}, we have $\Ext ^2_{A}(M,M)\cong H^2(A, \End _{k} (M))\cong H^2 (A, {_\rho M_n(k)_\rho})$. Thus the last assertion of the statement follows by Theorem \ref{teo:har} once proved the embedding of $\har^2(V_n(A),{_f k})$ into $H^2 (A, {_\rho M_n(k)_\rho})$. Let us construct it explicitly. The idea of the proof of this fact is inspired by \cite[Proposition 19.1.4]{G1} where the functor $M_n(-)$ is applied to a commutative extension with nilpotent kernel. Set $B:=V_n(A)$ and let $\omega:B\otimes B\rightarrow {_f k}$ be a Harrison $2$-cocycle. Consider the Hochschild extension associated to $\omega$
 \begin{equation}\label{Hochextk}
\xymatrix@C=30pt{ 0 \ar[r] & {_f k}\ar[r]^i &B_\omega \ar@<0.5ex>@{->}^-{p}[r]   & B\ar@<0.5ex>@{..>}^-{\sigma}[l]\ar[r] &0  }
\end{equation}
where, for brevity, we set $B_\omega:=B\oplus_\omega k$. %Explicitly, the multiplication and unit of $B_\omega$ are given by
 %\begin{gather*}
 % (b, k)\cdot(b', k'):=(bb',f(b')k+f(b)k'-\omega(b,b')),\quad 1_{B_\omega}=(1_B,\omega(1_B,1_B)).
 %\end{gather*}
 Set $S:=M_n(k)$ and apply the exact functor $S\otimes(-)$ to \eqref{Hochextk} to obtain the Hochschild extension
 \begin{equation}\label{HochextkS}
\xymatrix@C=40pt{ 0 \ar[r] & S\otimes{_f k}\ar[r]^{S\otimes i} &S\otimes B_\omega \ar@<0.5ex>@{->}^-{S\otimes p}[r]   & S\otimes B\ar@<0.5ex>@{..>}^-{S\otimes\sigma}[l]\ar[r] &0  }
\end{equation} Here $S\otimes{_f k}$ is a bimodule over $S\otimes B$ via $(s\otimes b)(s'\otimes l)(s''\otimes b'')=ss's''\otimes blb''=ss's''\otimes f(b)lf(b''),$ for every $s,s',s''\in S,l\in k, b,b''\in B$.  Now, let $E$ be either ${_f k}$, $B_\omega $ or $B$ and apply the canonical isomorphism $S\otimes E\to M_n(E):(k_{ij})\otimes e\mapsto (k_{ij}e)$ to \eqref{HochextkS} to obtain the Hochschild extension
 \begin{equation}\label{Hochextkn}
  \xymatrixrowsep{15pt}
 \xymatrix@C=30pt{  0 \ar[r] & N\ar[r]^-{i_n} &M_n(B_\omega) \ar@<0.5ex>@{->}^-{p_n}[r]   & M_n(B)\ar@<0.5ex>@{..>}^-{\sigma_n}[l]\ar[r] &0  }
\end{equation}
 where we set $p_n:=M_n(p),\sigma_n=M_n(\sigma)$, $i_n:=M_n(i)$ and $N$ is $M_n(k)$ regarded as a bimodule over $M_n(B)$ via $(b_{is})(l_{tj})=(\sum_s b_{is}l_{sj})$ and $(l_{tj})(b_{is})=(\sum_j l_{tj}b_{js})$ for every $(b_{is})\in M_n(B)$ and $(l_{tj})\in M_n(k).$ Thus$$(b_{is})(l_{tj})=(\sum_s b_{is}l_{sj})=(\sum_s f(b_{is})l_{sj})=f_n((b_{is}))\cdot(l_{tj})$$ where $f_n:=M_n(f)$ i.e. $N={_{f_n} (M_n(k))}$. Let $\eta=\eta_A:A\rightarrow M_n(V_n(A))=M_n(B)$ be the universal $n$-dimensional representation of Definition \ref{pigreco}. Hence $_\eta N={_\eta({_{f_n} (M_n(k))})}={_{(f_n\circ \eta )} M_n(k)}={_\rho M_n(k)}$, where we used that $f_n\circ \eta=M_n(f)\circ \eta=\rho$ which holds by definition of $\rho$. A similar argument applies to the right so that we get $_\eta N_\eta={_\rho M_n(k)_\rho}$.
 Let $\omega_n: M_n(B)\otimes M_n(B)\to N$ be the Hochschild $2$-cocycle associated to the Hochschild extension \eqref{Hochextkn}. Then $\overline{\omega_n}:=\omega_n\circ(\eta\otimes \eta)$ is a Hochschild $2$-cocycle so that we can consider the assignment $$\alpha:\har^2(V_n(A),{_f k})\to H^2 (A, {_\eta N_\eta}):[\omega] \mapsto [\overline{\omega_n}].$$
 This is a well-defined map. In fact, if $[\omega]=0$, then we can choose $\sigma$ to be an algebra map from the very beginning and hence $\sigma_n$ is an algebra map so that $\omega_n=0$. Suppose $\alpha([\omega])=0$. Then $\overline{\omega_n}$ is a $2$-coboundary. This condition guarantees, by the proof of Proposition \ref{pro:lifting}, that there is a $k$-algebra map $\overline{\lambda}:A\rightarrow M_n(B_\omega)$ such that $p_n\circ\overline{\lambda}=\eta_A.$ This map corresponds, via \eqref{proc}, to an algebra map $\lambda:B\rightarrow B_\omega$ such that $p\circ\lambda=\mathrm{Id}_B.$ This means that the Hochschild extension \eqref{Hochextk} is trivial whence $[\omega]=0.$ Thus $\alpha$ is injective.
\end{proof}

\begin{remark}\label{calabi}
The map $\har^2(V_n(A),{_f k}) \hookrightarrow H^2 (A, {_\rho M_n(k)_\rho})$ is not an isomorphism in general.
Furthermore the condition $\Ext_A^2(M,M)=0$ is not necessarily satisfied by regular points in $\ran.$
There is indeed the following counterexample.\\
Let $A$ be a $2$-Calabi Yau algebra, see \cite[Definition 3.2.3]{G2} for details.
It has been proven by Bocklandt that such an algebra has simple modules and that these modules are regular points in $\ran $ (see  \cite[Section 7.1]{Bock2}).
Therefore, for a simple  $M \in \ran (k)$ one has $ \har^2 (V_n(A),{_f k})=0.$
On the other hand, since $A$ is $2$-Calabi Yau, one has $\Ext ^2 _A (M, M)\cong \Ext ^0_A(M,M)\cong \End_A(M)\cong k, $ for all $M\in\ran.$

The referee pointed out to our attention the following example. Consider
the ($4$-dimensional) preprojective algebra $\Pi$ for a quiver of type $\textsf{A}_2$.
An elementary calculation shows that for the two 1-dimensional simple
$\Pi$-modules $S_1$ and $S_2$ one has $\Ext^1_\Pi(S_i, S_i) = 0$ but $\Ext^2_\Pi(S_i, S_i)\cong k$. Thus $\mathrm{Rep}_\Pi^n$ has two smooth points with non-trivial obstructions.
\end{remark}

\section{Examples and Applications}\label{exapp}

Next aim is to introduce and investigate the notion of \ps \ algebra. We will give several examples of such algebras. Moreover we will analyze the relationship between \ps, formally smooth and hereditary algebras to better understand the influence of the structure of $A$ on the smoothness of $\ran .$
\subsection{Finitely unobstructed algebras}
\begin{definition}\label{psdef}
Let $A$ be a $k$-algebra. Given $n\in \mathbb{N}$, we say that $A$ is $n$-\emph{\ps}, if $\Ext^2_{A}(M,M) =0,$  for  every $M\in \ran(k).$ We say that $A$ is \emph{\ps}, if it is $n$-\emph{\ps} for every $n\in \mathbb{N}$.
\end{definition}

\begin{corollary}\label{thm:ps}
The scheme $\ran$ is smooth for all $n$-\ps \ $k$-algebra $A$.
 \end{corollary}

 \begin{proof}
 It follows by Theorem \ref{thm:Hoch-Harr}.
 \end{proof}

We recast here some basic concepts in order to list examples and applications of the results proven in Section \ref{harri}.
\subsection{Hereditary algebras}
Recall that the \textit{projective dimension} $\mathrm{pd}(M)$ of an $M \in A\amod$ is the minimum length of a projective resolution of $M.\, $
\begin{definition}
The \textit{global dimension} of a ring $A,\,$ denoted with $gd(A)$, is the supremum of the set of projective dimensions of all (left) $A$-modules.
If $gd(A)\leq 1,$ then $A$ is called \textit{hereditary}.
\end{definition}

It holds that $gd(A)\leq d$ if and only if $\Ext^{d+1}_A(M,N)=0,$ for all $M,N\in A$-mod, see \cite[Proposition 2.1, page 110]{C-E}.

\subsection{Formally smooth algebras}\label{subsec:FS}
For further readings on these topics see \cite{G1,G2}.

\begin{definition}(Definition 3.3. \cite{C-Q}). \label{fsm}
An $A \in \alg _{k}$ is said to be \textit{formally smooth} (or  \textit {quasi-free}), if it satisfies the equivalent conditions:

\n
i)
any homomorphism $ \varphi \in \Hom_{\alg _{k}}(A,R/N)$ where $N$ is a nilpotent (two-sided) ideal in an algebra $R \in \alg _{k},\,$ can be lifted to a homomorphism $\overline{\varphi} \in \Hom_{\alg _{k}}(A,R)\,$ that commutes with the projection
$R \rightarrow R/N $;

\n
ii) $\hoch^2(A,M)=0$ for any $M\in A^e\amod$;

\n
iii) the kernel $\Omega ^1 _A$ of the  multiplication $A \otimes A \rightarrow A $ is a projective $A^e$-module.
\end{definition}

\begin{remark}
When $A$ is commutative $\Omega ^1 _A$ is nothing but the module of the K\"ahler differentials (see \cite[Section 8]{G1}).
\end{remark}

If we substitute $A \in \calg _{k}$ and $ \Hom_{\calg _{k}}(A,-)$ in Definition \ref{fsm}, we obtain the classical definition of regularity in the commutative case (see \cite[Proposition 4.1.]{LB}). On the other hand, if we ask for a commutative algebra $A$  to be formally smooth in the category $\alg _{k}$ we obtain regular algebras of dimension $\leq 1\,$ only (see \cite[Proposition 5.1.]{C-Q}). Thus, if $X=\mathrm{Spec} A$ is an affine smooth scheme, then $A$ is not formally smooth unless $\dim X \leq 1.\,$

\subsection{Implications and equivalences}
Let us collect the following, well-known, characterizations of finite-dimensional hereditary algebras.

\begin{theorem}\label{equiv}
Let $A$ be a finite-dimensional algebra over $k.\,$ The following assertions are equivalent:
 \begin{enumerate}
  \item[(1)] $A$ is formally smooth;
  \item[(2)] $\hoch^2(A,N)=0$ for every $N\in A^e\amodf$;
  \item[(3)] $A$ is \ps ;
  \item[(4)] $\mathrm{Rep}_{A}^{n}$ is smooth for every $n\in\mathbb{N}$;
  \item[(5)] $A$ is hereditary.
 \end{enumerate}
\end{theorem}
\begin{proof}
$(1)\Rightarrow(2)$ is trivial. Implication $(2)\Rightarrow(3)$ follows by Theorem \ref{hoch}, while $(3)\Rightarrow(4)$ follows by Corollary \ref{thm:ps}. Finally $(4)\Rightarrow(5)$ is \cite[Proposition 1]{Bo} and $(5)\Rightarrow(1)$ follows from \cite[Proposition 0.6.]{C-B3}.
\end{proof}

\begin{remark}\label{infi}
Let $A$ be an infinite-dimensional finitely generated algebra over $k.\,$ Let us check the following chain of implications
\[
(1)\Rightarrow (5) \Rightarrow (3)  \Rightarrow (4) \qquad\text{but}\qquad (1)\nLeftarrow (5) \nLeftarrow (3) \nLeftarrow (4)
\]
where the notations are the same of Theorem \ref{equiv}.

$(1)\Rightarrow (5).$ This is  \cite[Proposition 5.1]{C-Q}.

$(1)\nLeftarrow (5).$ Consider $A^1 = \CC[x,\delta]/<x\delta - \delta x =1> $ the first Weyl algebra. It is an example of a hereditary but not formally smooth algebra, since it can be proved that $\mathrm{H}^2(A,A^e)\neq 0\,$ (see \cite[Proposition 3]{VdB1}). This is due to William Crawley-Boevey (personal communication).

$(5)\Rightarrow (3).$ This is trivial.

$(5)\nLeftarrow (3).$ Let $U:=U(\mathfrak{g})$ be the universal enveloping algebra of a semisimple Lie algebra $\mathfrak{g}$. \emph{Whitehead's second lemma} (see e.g. \cite[Corollary 7.8.12, page 246]{We}) asserts that, in characteristic zero, $\mathrm{H}%
_{Lie}^{2}\left( \mathfrak{g},N\right) =0$ for every $\mathfrak{g}$-module $%
N $ of finite-dimension over $k $. In particular, for every $M\in A\amodf$, we obtain $\mathrm{H}_{Lie}^{2}\left( \mathfrak{g},\mathrm{End}_{k }(M)\right) =0$. By \cite[Exercise 7.3.5, page 226]{We}, we have
that
\begin{equation*}
\mathrm{H}_{Lie}^{\ast }\left( \mathfrak{g},\mathrm{End}_{k }\left(
M\right) \right) \cong \mathrm{Ext}_{U}^{\ast }\left(M ,M\right) .
\end{equation*}%
Therefore, $\mathrm{Ext}^2_U\left( M,M\right)=0$.
If $\mathfrak{g}$ is finite-dimensional, then $U$ is finitely generated and, thus, $U$ is \ps.
Now, $U$ has global dimension $\dim _{k }\left( \mathfrak{g}\right) ,$ see \cite[Exercise 7.7.2, page 241]{We}, and we are done.

$(3)\Rightarrow (4).$ This is Corollary \ref{thm:ps}.

$(3)\nLeftarrow (4).$ Remark \ref{calabi} shows that there might exist regular points $M$ in $\ran$ with  $\Ext_A^2(M,M)\neq 0. $
\end{remark}

\begin{remark}
The implication $(1)\Rightarrow (4)$ was already known, see \cite[Proposition 19.1.4.]{G1}, \cite[Prop.6.3.]{LB}.

Remark \ref{infi} and the argument on $U=U(\mathfrak{g}),$ contained in the proof thereby, together imply smoothness of $\mathrm{Rep}_{U}^{n}.$ This result was known, see e.g.  the comment by Le Bruyn in \cite{LB-forum}.
\end{remark}

\subsection{Unobstructed Algebras}
We now list some examples and results in case $A$ is finitely generated but not necessarily finite-dimensional.

In the remaining part of the section $k$ can be any field.

\begin{example}\label{ex:U}
We have seen in the Remark \ref{infi} that $U(\mathfrak{g})$ is \ps\  for a semisimple Lie algebra $\mathfrak{g}$.

More generally, in \cite[Theorem 0.2]{Zu}, there is a characterization of
all finite-dimensional Lie algebras $\mathfrak{g}$ over a field $k $ of
characteristic zero such that their second cohomology with coefficients in any
finite-dimensional module vanishes. Such a Lie algebra is one of the
following: (i) a one-dimensional Lie algebra; (ii) a semisimple Lie algebra;
(iii) the direct sum of a semisimple Lie algebra and a one-dimensional Lie
algebra. Note that a one-dimensional Lie algebra $\mathfrak{g}$ is not
semisimple as $[\mathfrak{g},\mathfrak{g}]=0\neq \mathfrak{g}$ (cf. \cite[%
Corollary at page 23]{Hu}). The same argument as above shows that the universal enveloping algebras of all of these Lie algebras are \ps.
\end{example}

The proof of the following result is analogous to \cite[Proposition 5.3(4)]{C-Q}. %For the reader sake, we include it here.

\begin{proposition}\label{pro:tens}
 Let $A$ and $S$ be \ps{} algebras over a field $k$. If $\Ext_S^1(M,M)=0$ for every $M\in S\amodf$, then $S\otimes A$ is \ps.
\end{proposition}

\begin{remark} \
Since $gd\, k[x_1,\dots,x_n]=n,$ from Remark \ref{infi} it follows that the algebra $k[x_1,\dots,x_n]$ is not formally smooth for $n > 1.$ \\
In general, the tensor product of two formally smooth algebras is not formally smooth. Indeed, in the setting of Proposition \ref{pro:tens}, if both $A$ and $S$ are finitely generated algebras over $k$, then, by \cite[Proposition 7.4]{C-E}, we have $\pd(S\otimes A)=\pd(S)+\pd(A)$, where $\pd(\Lambda)$ denotes the projective dimension of a $k$-algebra $\Lambda$ regarded as a bimodule over itself. Since $\pd(A)\leq n$ if and only if $\hoch^{n+1}(A,N)=0$ for every $N\in A^e\amod$, we get that the algebra $S\otimes A$ is not formally smooth unless $\pd(S)+\pd(A)\leq 1$ i.e. unless $S$ and $A$ are both formally smooth and at least one of them is separable.
\end{remark}

By using Proposition \ref{pro:tens}, we can give new examples of algebras whose associated representation scheme is smooth.

\begin{example}\label{psex}

1) Let $A$ be a \ps{} algebra and let $S$ be a separable algebra (see \cite[above Proposition 3.2]{C-Q}), that is  $\hoch^i (S, N)=0$ for every $i>0$ and for every $S$-bimodule $N$. By Theorem \ref{hoch}, we get $\mathrm{Ext}_S^i (M, M)=0$ for every $i>0$ and for every $M\in S\amod$. By Proposition \ref{pro:tens}, we get that $S\otimes A$ is \ps. As a particular case, when $\mathrm{char}(k)=0,$ we have that $M_n(A)\cong M_n(k)\otimes A$ and the group $A$-ring $A[G]\cong k[G]\otimes A $, for every finite group $G$, are \ps{} as the matrix ring $M_n(k)$ and the group algebra $k[G]$ are separable in characteristic zero (see \cite[Example of page 271]{C-Q}).

2) Let $A$ be \ps  \ algebra and $S$ a separable algebra, then $\mathrm{Rep}_{A \otimes S}^n$ is smooth. This follows from Proposition \ref{pro:tens} and example 1).

3) Let $\mathfrak{g}$ be a semisimple Lie algebra and assume $\mathrm{char}(k)=0$. As observed in Example \ref{ex:U}, $U:=U(\mathfrak{g})$ is \ps. Moreover Whitehead's first lemma \cite[Corollary 7.8.10]{We} ensures that $\hoch_{Lie}^{1}\left( \mathfrak{g},N\right) =0$ for every $\mathfrak{g}$-module $N $ of finite dimension over $k $ so that, by the same argument used in Example \ref{ex:U} for the second group of cohomology, we obtain $\mathrm{Ext}_U^1(M,M)=0$, for every $M\in U\amodf$. Thus, by Proposition \ref{pro:tens}, we get that $U\otimes A$ is \ps{} if $A$ is.

4) In analogy with \cite[Proposition 5.3(5)]{C-Q}, we have that the direct sum of \ps{} algebras is \ps{} too.
\begin{comment}
  In fact, by Theorem \ref{hoch}, we have to check that $\mathrm{H}^2 (A, \End _{k} (M))=0$ for every  $A^e$-module $M$ which is finite-dimensional over $k$. Now, in view of \cite[Theorem 5.3, page 173]{C-E}, we have $\mathrm{H}^2 (A, \End _{k} (M))\cong \mathrm{H}^2 (A_1, e_1\End _{k} (M)e_1)+\mathrm{H}^2 (A_2, e_2\End _{k} (M)e_2)$ where $e_1=(1,0),e_2=(0,1)$. Note that, for $i=1,2$, $e_i\End _{k} (M)e_i\cong\End _{k} (e_iM)$ and $e_iM$ is a subspace of $M$ whence finite-dimensional. Thus $\mathrm{H}^2 (A_i, e_i\End _{k} (M)e_i)\cong\mathrm{H}^2 (A_i, \End _{k} (e_iN))=0$ as $A_i$ if \ps.
\end{comment}
\end{example}

\begin{lemma}\label{lem:Omegaproj}
 Assume that $\hoch^2(A,N)=0$ for some $N\in A^e\amod$ and let $\Omega^1_A$ be as in Definition \ref{fsm}. Then $\Omega^1_A$ is projective with respect to any surjective morphism of $A^e$-modules with kernel $N$.
\end{lemma}

\begin{proof} One gets that $\mathrm{Ext}^1_{A^e}(\Omega^1_A,N)\cong\mathrm{Ext}^2_{A^e}(A,N)=\hoch^2(A,N)=0$ analogously to \cite[Proposition 3.3]{C-Q}. The conclusion follows by applying the long exact sequence of $\mathrm{Ext}^*_{A^e}(\Omega^1_A,-)$ to the exact sequence formed by any surjective morphism of $A^e$-modules and its kernel $N$.
\begin{comment}
  Set $\Omega:=\Omega^1_A$. If we write the long exact sequence of $\mathrm{Ext}^*_{A^e}(-,N)$ obtained from the short exact sequence $0\to \Omega\to A^e\to A\to 0$, and use that $\mathrm{Ext}^*_{A^e}(A^e,N)=0$, we easily get that $\mathrm{Ext}^1_{A^e}(\Omega,N)\cong\mathrm{Ext}^2_{A^e}(A,N)=\hoch^2(A,N)$. Therefore $\mathrm{Ext}^1_{A^e}(\Omega,N)=0$. Consider now a short exact sequence of $A^e$-modules of the form $0\to N\to P\overset{p}{\to} Q\to 0$ and let $f:\Omega\to Q$ be any morphism of $A^e$-modules
\begin{equation*}
 \xymatrixrowsep{15pt}
\xymatrix@C=30pt{&&&\Omega\ar^{f}[d]\ar@{.>}_{h}[dl]\\ 0 \ar[r] & N\ar[r]&P \ar^-{p}[r]   & Q\ar[r] &0  }
\end{equation*}
The latter exact sequence is turned into the exact sequence $$0\to \mathrm{Hom}_{A^e}(\Omega,N)\to\mathrm{Hom}_{A^e}(\Omega,P)\overset{p^*}{\to}\mathrm{Hom}_{A^e}(\Omega,Q)\to\mathrm{Ext}^1_{A^e}(\Omega,N)=0$$ where $p^*:=\mathrm{Hom}_{A^e}(\Omega,p):h\mapsto p\circ h$. Thus $p^*$ is surjective i.e. there is a morphism of $A^e$-modules $h:\Omega\to P$ such that $p\circ h=f$.
\end{comment}
\end{proof}

The proof of the following result is similar to \cite[Proposition 5.3(3)]{C-Q}.

\begin{proposition}\label{pro:tensAlg}
Let $A$ be a \ps{} algebra over a field $k$. Then the tensor algebra $T_A(P)$ is \ps{} for every $P\in  A^e\amod$ which is projective with respect to any surjective morphism in $A^e\amod$ with kernel $\End _{k} (M)$ for every $M\in A\amodf$.
\end{proposition}

\begin{comment} Let $M$ be an arbitrary left $T_A(P)$-module which is finite-dimensional over $k$ and set $N=\End _{k} (M)$. It suffices to prove that any Hochschild extension of $T_A(P)$ with kernel $N$ is trivial. Consider such an extension:
\begin{equation*}
 \xymatrixrowsep{15pt}
\xymatrix@C=30pt{&&& A\ar^{\sigma_A}[d]\ar@{.>}_{\overline{\sigma_A}}[dl]
\\ 0 \ar[r] & N\ar[r]&E \ar@<0.5ex>^-(.7){p}[r]   & T_A(P)\ar[r]\ar@<0.5ex>@{..>}^-(.3){\sigma}[l] &0
\\ &&&P\ar_{\sigma_P}[u]\ar@{.>}^{\overline{\sigma_P}}[ul]}
\end{equation*}
Since $A$ is \ps, by Proposition \ref{pro:lifting}, we have a lifting $\overline{\sigma_A}$ of the canonical inclusion $\sigma_A$. By means of $\overline{\sigma_A}$ we have that $E$ becomes an $A^e$-module and $p$ an $A^e$-module map. Since $M$ is also a left $A$-module through $\sigma_A$, then the projectivity of $P$ yields an $A^e$-module map $\overline{\sigma_P}$, as in the diagram above, such that $p\circ \overline{\sigma_P}=\sigma_P$, the latter being the canonical inclusion of $P$. Now, by the universal property of the tensor algebra, there is a unique algebra map $\sigma:T_A(P)\to E$ such that $\sigma\circ\sigma_A=\overline{\sigma_A}$ and $\sigma\circ\sigma_P=\overline{\sigma_P}$. It is easy to check that $p\circ\sigma=\id_{T_A(P)}$.
\end{comment}

\begin{example}Let $A$ be a \ps{} algebra over a field $k$. By Theorem \ref{hoch}, Lemma \ref{lem:Omegaproj} and Proposition \ref{pro:tensAlg}, we have that $T_A(\Omega^1_A)$ is \ps. The latter,
by \cite[Proposition 2.3]{C-Q}, identifies with $\Omega A,$ the DG-algebra of noncommutative
differential forms on $A.$
\end{example}

\section{Deformations}\label{defsmooth}
In this section we would like to analyze the relationships between the results of Section \ref{harri} and the theory of deformations of module structures.

\begin{definition}
Let $M \in \ran (k)$ and let  $\mu:A\to\End_k(M)$ be the associated linear representation. For $(R,\mathfrak{m})$  a local commutative $k$-algebra, an $R$-{\it
deformation}  of $M $ is an element $\widetilde{M}\in \ran(R) $ whose associate linear representation  $\widetilde{\mu }:  A\to \End_R(\widetilde{M})$ verifies $\alpha \circ \widetilde{\mu} = \mu$
where $\alpha :\End_R(\widetilde{M}) \to \End_k(M) $ is the morphism of $k$-algebras induceded by the projection $R\to R/\mathfrak{m}\cong k$.\\
When $R=k[\epsilon]:=k[t]/(t^2),$ the ring of dual numbers or $R=k[[t]],$ the ring of formal power series, then an $R$-deformation will be called \textit{infinitesimal} or \textit{formal}, respectively.
\end{definition}
For the general theory on deformations of finite-dimensional modules see  \cite{Ge} and \cite{Ger}.

\begin{remark}\label{obstructions}
It is well-known that the obstructions in extending the infinitesimal deformations of $M$ to formal deformations are in $\Ext^2_{A}(M,M)\,$ (see for example \cite[3.6. and 3.6.1.]{Ge}).
\end{remark}

The theory of local and global deformations of algebraic schemes is an ample and well-established domain of modern algebraic geometry.  Sernesi wrote an excellent treatise on this topic \cite{Se}, and we address the interested reader to it.

We just recall some facts we need to develop our analysis .\\
Let $X$ be a scheme over $ k,$  let $x \in X(k)$ be a $k$-point of $X$ and let $(R,\mathfrak{m})$ be a local commutative $k$-algebra.
\begin{definition}
An $R$-deformation of $X$ at $x$ is an $R$-point $x_R$ of $X$ such that the restriction $\mathrm{Spec}\, k\to \mathrm{Spec}\,R$ maps $x$ to $x_R.$
When $R=k[\epsilon]$  or $R=k[[t]]$, then an $R$-deformation will be called \textit{infinitesimal} or \textit{formal}, respectively.
\end{definition}

\begin{lemma}\label{locdef}
Let $R\in\calg_k$ and let $x:R\to k$ be a rational point of $X=\mathrm{Spec}\, R.$ Then, for all local $S\in\calg_k,$  there is a bijection
\[
\{S-{\mbox{deformations of }}X \mbox{ at } x\}\cong
\Hom_{\calg_k}(R_x, S)
\]
where $R_x$ denotes the localization of $R$ at $\mathfrak{m}_x:=\ker x.$
\end{lemma}
\begin{proof}
Let $\alpha: R\to S$ be such that $x=\pi_S\, \alpha,$ where $\pi_S:S\to S/\mathfrak{m}_S \cong k$ is the canonical projection with $\mathfrak{m}_S$ the maximal ideal of $S.$ Then if $a\in R-\ker x,$  it follows that $\alpha(a)$ is invertible in $S$ and, therefore, by universality, there is a unique morphism $\alpha_x:R_x\to S$ such that $\alpha_x\, j_R^x=\alpha$, where  $j_R^x:R\to R_x$ is the canonical map, and hence $x=\pi_S\,\alpha_x\, j^x_R$.\\
On the other hand, given a morphism $\beta:R_x\to S$ one has that $x=\pi_S\, \beta\, j_R^x$ thus giving the unique $S$-deformation $\beta\, j_R^x$ of $X$ at $x.$ It is, indeed, trivial that $\ker x\subset \ker (\pi_S\, \beta\, j_R^x).$
If $r\in \ker (\pi_S\, \beta\, j_R^x)$ then $\beta(j_R^x(a))\in \mathfrak{m}_S$ and, therefore, $j_R^x(a)\in \mathfrak{m}_{R_x}$. Thus $a\in\ker x.$
\end{proof}

The adjunction in Lemma \ref{rep} gives the dictionary to describe deformations of $A$-modules in terms of deformations at points of $\ran.$
The following result complements Theorem \ref{thm:Hoch-Harr}.

\begin{proposition}
Let $A\in\alg_k$ be finitely generated. Let $M \in \ran (k)$ and let  $f:\van\to k$ be the associated point. Then $M$ is regular if and only if,
for all finite-dimensional local commutative $k$-algebras $S,T$, a surjective homomorphism of $k$-algebras $S\to T$ induces a surjection
\[
\{S-{\mbox{deformations of }}M\}\longrightarrow
\{T-{\mbox{deformations of }}M\}.
\]
\end{proposition}

\begin{proof}
This follows for example from \cite[Proposition, pag 151]{M}.
\end{proof}

\begin{remark}
Geiss and de la Pe\~{n}a proved that, if $A$ is a finite-dimensional algebra, then  $M\in \ran(k)$ is regular  if $\Ext^2_A(M,M)=0,$ see \cite{Ge,Ge-P}. A careful analysis of their argument shows that it is easy to adapt their proof if one supposes that $A$ is bimodule coherent in the sense of \cite[3.5.1]{G2}, since in this case each finite-dimensional A-module admits a projective resolution by finitely generated projectives. If A is finitely presented, an argument involving cones, similar to \cite[Lemma 4.3]{CBS}, allows to extend the above mentioned argument to this situation.
\end{remark}

\bigskip
\centerline{\textbf{Acknowledgement}}
We would like to thank Corrado De Concini, Victor Ginzburg and Edoardo Sernesi for hints and very useful observations. Galluzzi and Vaccarino warmly thank  Sandra Di Rocco for the invitation to the KTH Department of Mathematics. Our gratitude also goes to the referees for their thorough reports that helped us to improve an earlier version of our paper.
\bibliographystyle{amsplain}

\bigskip

\begin{flushleft}

%{\bf AMS Subject Classification: 14J28, 14J50}\\[2ex]

Alessandro~Ardizzoni\\
Dipartimento di Matematica, Universit\`a di Torino, Via Carlo Alberto n.10, Torino, I-10123, ITALIA \\
e-mail: \texttt{alessandro.ardizzoni@unito.it}\\
URL: \url{sites.google.com/site/aleardizzonihome}\\[2ex]

Federica~Galluzzi\\
Dipartimento di Matematica, Universit\`a di Torino, Via Carlo Alberto n.10, Torino, 10123, ITALIA \\
e-mail: \texttt{federica.galluzzi@unito.it}\\[2ex]

Francesco~Vaccarino\\
Dipartimento di Scienze Matematiche, Politecnico di Torino, C.so Duca degli Abruzzi n.24, Torino, 10129, \ ITALIA \\
e-mail: \texttt{francesco.vaccarino@polito.it}\\
and\\
ISI Foundation, Via Alassio 11/c, 10126 Torino - Italy\\
e-mail: \texttt{vaccarino@isi.it}

\end{flushleft}

\end{document}